\documentclass[11pt, letterpaper, oneside]{amsart}
\usepackage{tikz}
\usepackage{hyperref}
\usepackage{amsrefs}
\usepackage{amsthm}
\usepackage{amssymb}
\usepackage{mathtools}
\usepackage{genyoungtabtikz}
\usepackage{float}

\newtheorem{theorem}{Theorem}[section]

\newtheorem{lemma}[theorem]{Lemma}
\newtheorem{proposition}[theorem]{Proposition}
\newtheorem{conjecture}[theorem]{Conjecture}

\theoremstyle{definition}

\newtheorem{example}[theorem]{Example}

\numberwithin{equation}{section}

\begin{document}

\title[Proofs of two perimeter inequalities]{A perimeter analogue of Franklin's identity and an inequality related to the parity of parts}
\author{Philip Cuthbertson}
\address{Department of Mathematical Sciences\\ Michigan Technological University\\ Houghton, MI 49931} \email{pecuthbe@mtu.edu}

\begin{abstract}
    We prove two conjectures regarding partition perimeter inequalities. It was conjectured by Gray, Payne, and Watson that Franklin's partition identity becomes an eventual inequality if one replaces the size of the partition with its perimeter. We prove this conjecture by asymptotic analysis of the corresponding generating functions. Gray, Payne, Swisher, and Watson also conjectured that there is a bias for partitions with fixed perimeter to have more odd parts than even parts. We prove this conjecture by deriving the corresponding generating functions and give a positive recurrence for the coefficients. In the case that there are an equal number of odd parts and even parts we provide a connection to peakless Motzkin paths.
\end{abstract}

\maketitle

\section{Introduction}

In this paper we prove two conjectures regarding perimeter analogues of known partition theorems. We note that during the preparation of this manuscript, after announcement of its subject as a talk at the October AMS Eastern Sectional Meeting, both of these conjectures have been proven by different authors entirely independent of the current work (See \cite{Mccrea}, \cite{Mahanta}). Conjecture \ref{Franklin Conjecture} is the perimeter analogue of Franklin's identity which becomes an inequality instead of equality and Conjecture \ref{Parity Conjecture} is a perimeter analogue of a theorem of Banerjee-Bhattacharjee-Dastidar-Mahanta-Saikia regarding a bias in the number of partitions with distinct parts with an extra constraint on how often even parts and odd parts can show up. Our proof of Conjecture \ref{Franklin Conjecture} is similar to that of \cite{Mccrea}, but we derive a quadrivariate generating function instead of a bivariate. Our proof of Conjecture \ref{Parity Conjecture} is entirely different than that given in \cite{Mahanta}. We derive generating functions for the series involved, give an exact formula for the coefficients, establish a recursion, and prove a connection to Motzkin paths. We believe there is independent value in our approaches to both conjectures.

Define $D_{j, k}(n)$ as the number of partitions of size $n$ with exactly $j$ distinct part sizes that appear at least $k$ times and $O_{j, k}(n)$ as the number of partitions of size $n$ with exactly $j$ distinct part sizes divisible by $k$. Franklin's theorem says that these two values coincide.

\begin{theorem}[Franklin's Identity]
    For $n, j \geq 0$ and $k \geq 2$,
    \[D_{j, k}(n) = O_{j, k}(n).\]
\end{theorem}

In \cite{Gray-Payne-Watson} the authors conjectured that Franklin's Identity becomes an inequality if we replace the size of a partition with its perimeter and proved the case of $k = 2$. If we let $FD_{j, k}(n)$ and $FO_{j, k}(n)$ be the analogous functions by perimeter equal to $n$ then they gave the following conjecture.

\begin{conjecture}[Gray-Payne-Watson]\label{Franklin Conjecture}
    For integers $j \geq 0$ and $k \geq 2$, there is some $N$ such that for all $n \geq N$ we have
    \[FD_{j, k}(n) \geq FO_{j, k}(n).\]
\end{conjecture}

We derive quadrivariate generating functions for the partitions counted by $FD$ and $FO$ and then prove the conjecture using ``singularity analysis'' in the sense of Flajolet-Sedgewick. 

We define $pd_o(n)$ as the number of partitions of size $n$ into distinct parts where there are more odd parts than even parts and $pd_e(n)$ is the number of partitions of size $n$ into distinct parts where there are more even parts than odd parts. In \cite{BBDMS}, Banerjee, Bhattacharjee, Dastidar, Mahanta, and Saikia proved the following theorem regarding the bias of these two counts.

\begin{theorem}[Banerjee-Bhattacharjee-Dastidar-Mahanta-Saikia \cite{BBDMS}]
    For $n \geq 20$,
    \[pd_o(n) > pd_e(n).\]
\end{theorem}

In \cite{Gray-Payne-Swisher-Watson} Gray, Payne, Swisher, and Watson conjectured that a similar inequality holds for the corresponding functions with size replaced by perimeter.

\begin{conjecture}[Gray-Payne-Swisher-Watson]\label{Parity Conjecture}
    For all $n \geq 9$,
    \[rd_o(n) > rd_e(n).\]
\end{conjecture}

In fact, $rd_o(n) \geq rd_e(n)$ for all $n \neq 2$. We prove this conjecture by establishing a recurrence relation for the difference of these two quantities and then the positivity comes from induction. In \cite{Mahanta}, this was proven by defining an injection from one side of the inequality to the other. We also derive a formula for these functions in terms of a binomial sum and draw a connection to Motzkin paths.

In section 2 we recall some basics about integer partitions, a few analytic tools we will be using, and the notation we will be using, section 3 will be dedicated to a proof of Conjecture \ref{Franklin Conjecture}, and section 4 will be for the proof of Conjecture \ref{Parity Conjecture}.

\section{Background}

An \textit{integer partition}, $\mu = (\mu_1, \mu_2, \ldots, \mu_m)$, is a weakly decreasing sequence of positive integers. Each of the $\mu_i$ will be referred to as the \textit{parts} of $\mu$ and $\ell(\mu) := m$ is called the length. Throughout we are primarily concerned with the \textit{perimeter} of a partition which is defined by $P(\mu) := \mu_1 + \ell(\mu) - 1$ and has been a very active topic of study in partition theory. We will also make use of the \textit{Ferrers diagram} of a partition which is stack of left justified boxes such that the $i$th row has $\mu_i$ boxes. Since the perimeter of a partition is often defined in terms of a hook length, we occasionally refer to $\mu_1$ as the arm length and $\ell(\mu)$ as the leg length of the perimeter. In the proof of Conjecture \ref{Franklin Conjecture} we will be relying on the following theorem that gives an asymptotic expansion for rational functions in terms of their dominant singularities and a well known inequality.

\begin{theorem}[Corrected from IV.9 / IV.26 in \cite{Flajolet-Sedgewick}]\label{Asymptotic Formula}
    If $f(z)$ is a rational function that is analytic at zero and has a single dominant pole $\alpha$, with multiplicity $r$, then
    \[[z^n]f(z) = C\left(\frac{(-1)^r n^{r - 1}}{(r - 1)! \alpha^{n + r}}\right)\left(1 + O\left(\frac{1}{n}\right)\right)\]
    where 
    \[C :=\lim_{z\to\alpha}(z-\alpha)^r f(z).\]
\end{theorem}

\begin{lemma}[Bernoulli's Inequality]\label{Bernoulli}
    For $r \geq 1$ and $x \geq -1$,
    \[(1 + x)^r \geq 1 + rx.\]
\end{lemma}

Later we draw a connection to peakless Motzkin paths so we recall their definitions here. A \textit{Motzkin path} of length $n$ is a lattice path from $(0, 0)$ to $(n, 0)$ that only takes steps $U = (1, 1), H = (1, 0),$ and $D = (1, -1)$ and does not travel below the $x$-axis. A \textit{peakless Motzkin path} is a Motzkin path that has no $U$ step that is immediately followed by a $D$ step. With slight abuse of notation we will often treat the Motzkin path and the word of $U$'s, $H$'s, and $D$'s naturally corresponding to the path synonymously. We will also need the following lemma.

\begin{lemma}[Raney's Lemma \cite{Raney}, Section 7.5 in \cite{ConcreteMathematics}]\label{Raney}
    If $a_1, a_2, \ldots a_n$ is a sequence of integers whose sum is 1, then exactly one cyclic permutation of $a_1, a_2, \ldots a_n$ has every partial sum positive.
\end{lemma}

In particular we will need the following form which we prove is equivalent.

\begin{lemma}\label{Negative Raney}
    If $b_1, b_2, \ldots b_n$ is a sequence of integers whose sum is $-1$, then exactly one cyclic permutation of $b_1, b_2, \ldots b_n$ has every strict partial sum non-negative.
\end{lemma}

\begin{proof}
    Let $(a_1, a_2, \ldots, a_n)$ be a sequence of integers whose total sum is 1 and define a new sequence $(b_1, b_2, \ldots, b_n)$ by $b_i = -a_{n + 1 - i}$. Notice that
    \[\sum_{i = 1}^n b_i = \sum_{i = 1}^n -a_{n + 1 - i} = - \sum_{i = 1}^n a_i = -1.\]
    By Lemma \ref{Raney}, we can cyclically shift $(a_1, a_2, \ldots, a_n)$ in a unique way such that each partial sum is positive. Let $(a_1', a_2', \ldots, a_n')$ be that rotation so that $b_i' = -a_{n + 1 - i}'$. Notice that for $1 \leq i \leq n$
    \[a_1' + a_2' + \cdots + a_i' > 0\]
    implies that
    \[b_n' + b_{n-1}' + \cdots + b_{n + 1 - i}' = \sum_{i = 1}^n b_i' - (b_1' + b_2' + \cdots + b_{n - i}') < 0.\]
    The total sum is $-1$, so we obtain
    \[-1 < b_1' + b_2' + \cdots + b_i'\]
    which is the same as the strict partial sums being non-negative. The uniqueness of $(b_1', b_2', \ldots, b_n')$ follows from the uniqueness of $(a_1', a_2', \ldots, a_n')$. Clearly we can also go in the opposite direction to show that these two statements are equivalent.
\end{proof}

\section{Proof of Conjecture \ref{Franklin Conjecture}}

We will first derive the corresponding generating functions using the same kind of technique used in Gray-Payne-Watson \cite{Gray-Payne-Watson} which itself is similar to that of Fu-Tang \cite{Fu-Tang}.

\begin{proposition}\label{FD Function}
    We have
    \begin{flalign*}
        \mathcal{FD}_k(z, q) &:= \sum_{j = 0}^\infty \sum_{n = 0}^\infty FD_{j, k}(n) z^j q^n =  \frac{q - (1 - z) q^k}{1 - 2q + (1 - z)q^{k + 1}}\\
        &= \frac{q - q^k}{1 - 2q + q^{k + 1}} + (1 - q)^2 \sum_{j = 1}^\infty \frac{q^{jk + j - 1}}{(1 - 2q + q^{k + 1})^{j + 1}}z^j.
    \end{flalign*}
\end{proposition}

\begin{proof}
    We begin by establishing a quadrivariate generating function and then specialize. Define
    \[\mathcal{FD}_k (x, y, z, q) := \sum_{j = 0}^\infty\sum_{\lambda = 0}^\infty\sum_{\alpha = 0}^\infty FD_{j, k}(\alpha, \lambda) x^\alpha y^\lambda z^j q^{\alpha + \lambda - 1}\]
    where $FD_{j, k}(\alpha, \lambda)$ is the number of partitions counted by $FD_{j, k}(n)$ where the perimeter has arm length $\alpha$ and leg length $\lambda$. We build a partition by keeping track of the East and North moves along the profile. We must start with an $E$ followed by some, potentially 0, amount of $N$'s noting that if there are at least $k$ steps North then we have to increment $z$. This contributes
    \[xq (1 + yq + (yq)^2 + \ldots + (yq)^{k - 1} + z(yq)^k + z(yq)^{k + 1} + \ldots) = \frac{xq(1 - (yq)^k (1 - z))}{1 - yq}.\]
    This block can be repeated any number of times, but the last block will be slightly different since it must have at least one $N$ move. The final block is generated by
    \[xyq (1 + yq + (yq)^2 + \ldots + (yq)^{k - 2} + z(yq)^{k - 1} + z(yq)^k + \ldots) = \frac{xyq(1 - (yq)^{k - 1} (1 - z))}{1 - yq}.\]
    All combined we get the quadrivariate generating function to be 
    \[\frac{xyq(1 - (yq)^{k - 1} (1 - z))}{1 - yq} \left( 1 + \frac{\frac{xq(1 - (yq)^k (1 - z))}{1 - yq}}{1-\frac{xq(1 - (yq)^k (1 - z))}{1 - yq}}\right).\]
    \[=\frac{xyq\left(1 - (yq)^{k - 1}(1 - z)\right)}{1 - yq - xq\left(1 - (yq)^k(1 - z)\right)}\]
    After setting $x = y = 1$ one gets
    \[\frac{q\left(1 - q^{k - 1}(1 - z)\right)}{1 - q - q \left(1 - q^k(1 - z)\right)}\]
    % \[\frac{q(1 - q^{k - 1} (1 - z))}{1 - q} \left( 1 + \frac{\frac{q(1 - q^k (1 - z))}{1 - q}}{1-\frac{q(1 - q^k (1 - z))}{1 - q}}\right).\]
    Now we expand as a geometric series to get the second equality.
    \begin{flalign*}
        \frac{q - (1 - z) q^k}{1 - 2q + (1 - z)q^{k + 1}} &= \frac{q - q^k + z q^k}{1 - 2q + q^{k + 1} - zq^{k + 1}}\\
        &= \frac{q - q^k + z q^k}{1 - \frac{zq^{k + 1}}{1 - 2q + q^{k + 1}}} \cdot \frac{1}{1 - 2q + q^{k + 1}} = \sum_{j = 0}^\infty \frac{(q - q^k + z q^k)q^{j(k + 1)}}{(1 - 2q + q^{k + 1})^{j + 1}}z^j\\
        &= \sum_{j = 0}^\infty \left(\frac{(q - q^k)q^{j(k + 1)}}{(1 - 2q + q^{k + 1})^{j + 1}}z^j + \frac{q^{(j + 1)(k + 1) - 1}}{(1 - 2q + q^{k + 1})^{j + 1}}z^{j + 1}\right)\\
        &= \frac{q - q^k}{1 - 2q + q^{k + 1}} \\
        &+ \sum_{j = 1}^\infty \left(\frac{(q - q^k)q^{j(k + 1)}}{(1 - 2q + q^{k + 1})^{j + 1}}z^j + \frac{q^{j(k + 1) - 1}(1 - 2q + q^{k + 1})}{(1 - 2q + q^{k + 1})^{j + 1}}z^{j}\right).
    \end{flalign*}
    % \textcolor{red}{Pulling out the $z^0$ term then recombining we get the proposition.}
    Combining the two summands and factoring gives the proposition.
\end{proof}

\begin{proposition}\label{FO Function}
    We have
    \begin{flalign*}
        \mathcal{FO}_k(z, q) &:= \sum_{j = 0}^\infty \sum_{n = 0}^\infty FO_{j, k}(n) z^j q^n = \frac{q(1-q)^k + q^{k+1}(1 - 2z) - q^k(1 - z)}{(1 - 2q)((1 - q)^k - q^k(1 - q + zq))}\\
        =& \frac{q((1-q)^{k-1}-q^{k-1})}{(1-2q)((1-q)^{k-1}-q^k)} \\
        &+ \sum_{j = 1}^\infty \frac{q^{jk+j-1}((1-q)^k-q^k)}{(1-2q)(1-q)^{j-1}((1-q)^{k-1}-q^k)^{j+1}}z^j.
    \end{flalign*}
\end{proposition}

\begin{proof}
    We begin by establishing a quadrivariate generating function and then specialize. Define
    \[\mathcal{FO}_k(x, y, z, q) := \sum_{j = 0}^\infty\sum_{\lambda = 0}^\infty\sum_{\alpha = 0}^\infty FO_{j, k}(\alpha, \lambda) x^\alpha y^\lambda z^j q^{\alpha + \lambda - 1}\]
    where $FO_{j, k}(\alpha, \lambda)$ is the number of partitions counted by $FO_{j, k}(n)$ where the perimeter has arm length $\alpha$ and leg length $\lambda$. For notational convenience define the following 
    \begin{flalign*}
        a &:= \frac{xq}{1 - yq}\\
        b &:= xq\left(1 + \frac{zyq}{1 - yq}\right)\\
        c &:= \frac{xyzq}{1 - yq}.
    \end{flalign*}
    Now similar to \ref{FD Function} we start with an $E$ followed by an arbitrary number of $N$'s, but now we need to keep track of the number of $E$'s. Our typical block will be counted by $a$ and we use a $b$ or a $c$ whenever we have a multiple of $k$ previous blocks. We use $c$ when this is the last block and use $b$ when it is not the last block. This is going to be generated by
    \begin{flalign*}
        (a + &\ldots + a^{k - 1}) + a^{k - 1} c + (a + \ldots + a^{k - 1}) a^{k - 1} b + a^{2(k - 1)} b c + \ldots \\
        &= \frac{a - a^k}{1 - a}\left(1 + a^{k - 1}b + a^{2(k - 1)}b + \ldots\right) + a^{k - 1} c + a^{2(k - 1)} bc + \ldots \\
        &= \frac{a - a^k}{1 - a}\left(\frac{1}{1 - a^{k - 1}b}\right) + ca^{k - 1}\left(\frac{1}{1 - a^{k - 1}b}\right)\\
        &= \frac{a + ca^{k - 1} - ca^k - a^k}{(1 - a)(1 - a^{k - 1}b)}.
    \end{flalign*}
    After setting $x = y = 1$, we obtain
    \begin{flalign*}
        &\frac{(\frac{q}{1 - q}) + \frac{zq}{1 - q}(\frac{q}{1 - q})^{k - 1} - \frac{zq}{1 - q}(\frac{q}{1 - q})^k - (\frac{q}{1 - q})^k}{(1 - (\frac{q}{1 - q}))(1 - (\frac{q}{1 - q})^{k - 1}q\left(1 + \frac{zq}{1 - q}\right))} \\
        &= \frac{q(1 - q)^k + zq \cdot q^{k - 1}(1 - q) - zq \cdot q^k - q^k(1 - q)}{(1 - 2q)((1 - q)^k - q^{k - 1}\cdot q\left(1 - q + zq\right))} \\
        % = \frac{q(1 - q)^k + zq^k - zq^{k + 1} - zq^{k + 1} - q^k + q^{k + 1}}{(1 - 2q)((1 - q)^k - q^k\left(1 - q + zq\right))} \\
        &= \frac{q(1 - q)^k - q^k + q^{k + 1} + zq^k(1 - 2q)}{(1 - 2q)((1 - q)^k - q^k + q^{k + 1} - zq^{k + 1})}
    \end{flalign*}
    regrouping the numerator gives the first equality in the proposition. Now we expand as a geometric series to get the second equality.
    \begin{flalign*}
        &\frac{q(1 - q)^k - q^k + q^{k + 1} + zq^k(1 - 2q)}{(1 - 2q)((1 - q)^k - q^k + q^{k + 1} - zq^{k + 1})} \\
        &= \frac{q(1 - q)^k - q^k + q^{k + 1} + zq^k(1 - 2q)}{1 - \frac{zq^{k + 1}}{(1 - q)^k - q^k + q^{k + 1}}} \cdot \frac{1}{(1 - 2q)((1 - q)^k - q^k + q^{k + 1})} \\
        &= \frac{1}{1 - 2q}\sum_{j = 0}^\infty \frac{(q(1 - q)^k - q^k + q^{k + 1} + zq^k(1 - 2q))q^{j(k + 1)}}{((1 - q)^k - q^k + q^{k + 1})^{j + 1}}z^j \\
        &= \frac{1}{1 - 2q}\sum_{j = 0}^\infty \left(\frac{(q(1 - q)^k - q^k + q^{k + 1})q^{j(k + 1)}}{((1 - q)^k - q^k + q^{k + 1})^{j + 1}}z^j + \frac{(1 - 2q)q^{(j + 1)(k + 1) - 1}}{((1 - q)^k - q^k + q^{k + 1})^{j + 1}}z^{j + 1}\right) \\
        &= \frac{q(1 - q)^k - q^k + q^{k + 1}}{(1 - 2q)((1 - q)^k - q^k + q^{k + 1})} + \\
        &\frac{1}{1 - 2q}\sum_{j = 1}^\infty \frac{(q(1 - q)^k - q^k + q^{k + 1})q^{j(k + 1)} + (1 - 2q)q^{j(k + 1) - 1}((1 - q)^k - q^k + q^{k + 1})}{((1 - q)^k - q^k + q^{k + 1})^{j + 1}}z^j.
    \end{flalign*}
    Expanding the numerator and factoring gives the result in the proposition.
\end{proof}

Using ``singularity analysis'' \cite{Flajolet-Sedgewick} we show that $FD_{j, k}(n)$ grows asymptotically faster than $FO_{j, k}(n)$ which implies Conjecture \ref{Franklin Conjecture}. If $k = 2$, the two series are equivalent so we only need to concern ourselves with $k \geq 3$.

\begin{proposition}\label{unique roots}
    For $k \geq 2$, the polynomials $f_k(q) := 1 - 2q + q^{k + 1}$ and $g_k(q) := (1 - q)^{k - 1} - q^k$ each have a unique simple root of minimal modulus that lies on the real interval $(1/2, 1)$.
\end{proposition}

\begin{proof}
    Fix $k \geq 2$. The Intermediate Value Theorem gives that $f_k$ and $g_k$ both have roots on that interval. Let $\alpha_k$ and $\beta_k$ be the minimal real roots of $f_k$ and $g_k$ respectively. Since $\alpha_k < 1$ we can cancel the factor of $1 - q$ in $f_k(q)$ and notice that $f_k(q)/(1 - q) = 1 - q - \cdots - q^k$ has a negative derivative on $(1/2, 1)$ so $\alpha_k$ is simple. Additionally, $g_k'(q) < 0$ for all $q \in (1/2, 1)$ so $\beta_k$ is also simple. For any other root $a$ of $1 - q - \cdots - q^k$, the triangle inequality implies
    \[1 = |a + \cdots + a^k| \leq |a| + \cdots + |a^k| \leq \alpha_k + \cdots + \alpha_k^k = 1\]
    and so $a$ must be a positive real number and thus equal to $\alpha_k$. For some $b$ of modulus at most 1 that is a root of $g_k(q)$, the reverse triangle inequality implies
    \[|b|^k = |1 - b|^{k - 1} \geq (1 - |b|)^{k - 1}.\]
    So $g_k(|b|) \leq 0$ and so $\beta_k \leq |b|$.
\end{proof}

\begin{proposition}
    Let $\alpha_k$ be the minimal positive root of $f_k(q) := 1 - 2q + q^{k + 1}$ and $\beta_k$ be the minimal positive root of $g_k(q) := (1 - q)^{k - 1} - q^k$ on the interval $(1/2, 1)$. For $j = 0$ and $k \geq 3$ we have the expansions
    \[FD_{0,k}(n) = \frac{(1 - \alpha_k)^2}{\alpha_k^2(2 - (k + 1)\alpha_k^k)}\left(\frac{1}{\alpha_k^n}\right)\left(1 + O\left(\frac{1}{n}\right)\right).\]
    and
    \[FO_{0,k}(n) = \frac{(1 - \beta_k)^k}{\beta_k(2\beta_k - 1)\left((k - 1)(1 - \beta_k)^{k - 2} + k \beta_k^{k - 1}\right)}\left(\frac{1}{\beta_k^n}\right)\left(1 + O\left(\frac{1}{n}\right)\right).\]
    For $j \geq 1$, we have
    \[FD_{j,k}(n) = \frac{(1 - \alpha_k)^2 \alpha_k^{jk - 2}}{(2 - (k + 1) \alpha_k^k)^{j + 1} j!}\left(\frac{n^j}{\alpha_k^n}\right)\left(1 + O\left(\frac{1}{n}\right)\right)\]
    and
    \[FO_{j,k}(n) = \frac{\beta_k^{jk + k - 1}}{(2\beta_k - 1) (1 - \beta_k)^{j-1} \left((k - 1)(1 - \beta_k)^{k - 2} + k \beta_k^{k - 1}\right)^{j + 1} j!}\left(\frac{n^j}{\beta_k^n}\right)\left(1 + O\left(\frac{1}{n}\right)\right).\]
\end{proposition}

\begin{proof}
    Fix $k \geq 3$. We note that $q = 1/2$ is a removable singularity in $\mathcal{FO}_k(z, q)$. The roots $\alpha_k$ and $\beta_k$ are unique by Proposition \ref{unique roots}. We begin with
    \[\frac{q - q^k}{1 - 2q + q^{k + 1}}\]
    which by Theorem \ref{Asymptotic Formula} gives the following asymptotic expansion
    \begin{flalign*}
        FD_{0,k}(n) &= \lim_{q \to \alpha_k} \frac{(q - \alpha_k)(q - q^k)}{1 - 2q + q^{k + 1}} \left(\frac{-1}{\alpha_k^{n+1}}\right)\left(1 + O\left(\frac{1}{n}\right)\right) \\
        &= \lim_{q \to \alpha_k} \frac{-(q - q^k) - (q - \alpha_k)(1 - kq^{k - 1})}{(k + 1) q^k - 2} \left(\frac{1}{\alpha_k^{n+1}}\right)\left(1 + O\left(\frac{1}{n}\right)\right) \\
        &= \frac{\alpha_k - \alpha_k^k}{2-(k + 1)\alpha_k^k}\left(\frac{1}{\alpha_k^{n+1}}\right)\left(1 + O\left(\frac{1}{n}\right)\right) \\
        &= \frac{(1 - \alpha_k)^2}{\alpha_k^2(2 - (k + 1)\alpha_k^k)}\left(\frac{1}{\alpha_k^n}\right)\left(1 + O\left(\frac{1}{n}\right)\right)
    \end{flalign*}
    Additionally, for $j \geq 1$ we have
    \begin{flalign*}
        FD_{j,k}(n) &= \lim_{q \to \alpha_k} \frac{(q - \alpha_k)^{j + 1}(1 - q)^2 q^{jk + j - 1}}{(1 - 2q + q^{k + 1})^{j + 1}} \left(\frac{(-1)^{j + 1} n^j}{\alpha_k^{n + j + 1} j!}\right)\left(1 + O\left(\frac{1}{n}\right)\right) \\
        &= \lim_{q \to \alpha_k} \frac{(j + 1)!(1 - q)^2 q^{jk + j - 1} + (q - \alpha_k)(\ldots)}{(j + 1)!((k+1)q^k - 2)^{j + 1}} \left(\frac{(-1)^{j + 1} n^j}{\alpha_k^{n + j + 1} j!}\right)\left(1 + O\left(\frac{1}{n}\right)\right) \\
        &= \frac{(1 - \alpha_k)^2 \alpha_k^{jk + j - 1}}{((k + 1) \alpha_k^k - 2)^{j + 1}} \left(\frac{(-1)^{j + 1} n^j}{\alpha_k^{n + j + 1} j!}\right)\left(1 + O\left(\frac{1}{n}\right)\right) \\
        &= \frac{(1 - \alpha_k)^2 \alpha_k^{jk - 2}}{(2 - (k + 1) \alpha_k^k)^{j + 1} j!}\left(\frac{n^j}{\alpha_k^n}\right)\left(1 + O\left(\frac{1}{n}\right)\right).
    \end{flalign*}
    The derivation of the other two formulas is nearly identical.
\end{proof}

We now show that, asymptotically, $FD_{j, k}(n)$ dominates by showing that $\alpha_k < \beta_k$ for all $k \geq 3$.

\begin{proposition}\label{alpha beta inequality}
    If $f_k(q) := 1 - 2q + q^{k + 1}$ and $g_k(q) := (1 - q)^{k - 1} - q^k$ and $\alpha_k$ is the smallest positive real root of $f_k(q)$ and $\beta_k$ the smallest positive real root of $g_k(q)$ on $(1/2, 1)$, then for all $k \geq 3$, we have $\alpha_k < \beta_k$.
\end{proposition}

\begin{proof}
    % Fix $k \geq 3$ and notice that $g_k(0) = 1$ and $g_k(1) = -1$ so by the intermediate value theorem, $\beta_k \in (0, 1)$. Also notice that $g_k(q)$ is decreasing on this interval since $g_k'(q) < 0$ for $q \in (0, 1)$ and thus so it is sufficient to show that $g_k(\alpha_k) > 0$.

    Fix $k \geq 3$. By taking the derivative, we see that $g_k(q)$ is decreasing on the interval $(0, 1)$ and by the Intermediate Value Theorem $\beta_k$ lies on this same interval. It is thus sufficient to show that $g_k(\alpha_k) > 0$. Notice that for $q \in (0, 1)$ we have
    \[f_k(q) < f_{k-1}(q) < \cdots < f_3(q)\]
    which implies
    \[\alpha_k < \alpha_{k-1} < \cdots < \alpha_3 = \frac{\sqrt[3]{586+102\sqrt{33}}-\sqrt[3]{17+3\sqrt{33}}-2}{3\sqrt[3]{17+3\sqrt{33}}} < 0.55.\]
    Noticing that $x(0.55)^{x - 1} < 1$ for $x$ at least 2.60, we obtain the following string of inequalities
    \[1 > k(0.55)^{k - 1} > k\alpha_3^{k - 1} \geq k\alpha_k^{k-1} > \alpha_k^k + (k - 1) \alpha_k^{k-1}.\]
    Comparing the extreme sides we obtain
    \[1 - \alpha_k^k > (k-1) \alpha_k^{k - 1}\]
    and now multiplying by $\alpha_k$ and using the defining relation $1 - 2\alpha_k + \alpha_k^{k + 1} = 0$, we obtain
    \[1 - \alpha_k > (k - 1)\alpha_k^k.\]
    Rearranging and applying Lemma \ref{Bernoulli}, we obtain
    \[(1 - \alpha_k^k)^{k - 1} \geq 1 - (k - 1) \alpha_k^k > \alpha_k.\]
    So
    \[(1 - \alpha_k^k)^{k - 1} - \alpha_k > 0\]
    and multiplying by $\alpha_k^{k - 1}$ gives
    \[(\alpha_k - \alpha_k^{k + 1})^{k - 1} - \alpha_k^k > 0.\]
    Once again using $1 - 2\alpha_k + \alpha_k^{k + 1} = 0$ gives
    \[g_k(\alpha_k) = (1 - \alpha_k)^{k - 1} - \alpha_k^k > 0.\]
\end{proof}

Clearly Proposition \ref{alpha beta inequality} implies the truth of Conjecture \ref{Franklin Conjecture}. We also use the quadrivariate generating functions to obtain the following Beck-type theorem.

% \begin{theorem}
%     Let $\Delta_{j,k} T(n)$ be the total number of parts in all partitions counted by $FD_{j,k}(n)$ minus the total number of parts in all partitions counted by $FO_{j,k}(n)$ and similarly let $\Delta_{j,k} L(n)$ be the sum of the largest parts in all partitions counted by $FD_{j,k}(n)$ minus the sum of the largest parts in all partitions counted by $FO_{j,k}(n)$. Then we have for all $j, k$ and $n$ that
%     \[\Delta_{j,k} T(n) + \Delta_{j,k} L(n) = (n + 1) (FD_{j,k}(n) + FO_{j,k}(n)).\]
% \end{theorem}

% \begin{proof}
%     \textcolor{red}{Basically just comes down to length + largest part = perimeter + 1}
% \end{proof}

\section{Proof of Conjecture \ref{Parity Conjecture}}

    In the proof of Theorem 2.1 in \cite{Straub} Straub used induction to show that a partition into distinct parts can be built by adding 1 onto the largest part or by appending a new part to the partition that is exactly one larger than the current largest part. We use this idea to find a trivariate generating function for partitions into distinct parts while keeping track of the number of odd parts, the number of even parts, and the perimeter simultaneously.

\begin{theorem}
    If $o(\lambda)$ is the number of odd parts in $\lambda$ and $e(\lambda)$ the number of even parts in $\lambda$, 
    \[F(x, y, q) := \sum_{\lambda \in \mathcal D}x^{o(\lambda)}y^{e(\lambda)} q^{per(\lambda)} = \frac{xq + yq^2 + xyq^3}{1 - q^2(1 + xq)(1 + yq)}.\]
\end{theorem}

\begin{proof}
    Let $O(x, y, q)$ be the generating functions where $x$ keeps track of the number of odd parts, $y$ keeps track of the number of even parts, $q$ keeps track of the perimeter of the partition, and where the largest part is odd. Similarly, $E(x, y, q)$ is the same except the largest part is even. Clearly $F(x, y, q) = O(x, y, q) + E(x, y, q)$. We are going to build a recursion with the two operations of adding one to the largest part or appending a new part one larger than the current largest part. Both of these change the parity of the largest part. Adding one to the largest part increases the perimeter by one and adding on the new largest part increases the perimeter by two. Thus we get the two relations
    \begin{flalign*}
        O(x, y, q) = xq + \left( \frac{xq}{y} + xq^2 \right) E(x, y, q) \\
        E(x, y, q) = \left( \frac{yq}{x} + yq^2 \right) O(x, y, q)
    \end{flalign*}
    where the term of $xq$ is used to keep track of the partition $(1)$. It is routine to then uncouple these to obtain
    \begin{flalign*}
        O(x, y, q) = \frac{xq}{1 - q^2(1 + xq)(1 + yq)} \\
        E(x, y, q) = \frac{yq^2 + xyq^3}{1 - q^2(1 + xq)(1 + yq)}.
    \end{flalign*}
    Adding these together proves the theorem.
\end{proof}

\begin{theorem}
    For all $m$,
    \begin{flalign*}
        rd_o(2m) - rd_e(2m) &= \sum_{n = 0}^\infty \binom{n}{m - n - 1}\left(\binom{n}{m -n} - \binom{n}{m - n - 1}\right) \\
        rd_o(2m + 1) - rd_e(2m + 1) &= \sum_{n = 0}^\infty \binom{n}{m - n}\left(\binom{n}{m - n}-\binom{n}{m - n - 1}\right).
    \end{flalign*}
\end{theorem}

\begin{proof}
    We can specialize $x$ and $y$ to obtain the following

\[F(z, z^{-1}, q) = \sum_{\lambda \in \mathcal D}z^{o(\lambda) - e(\lambda)} q^{per(\lambda)} = \frac{zq + z^{-1}q^2 + q^3}{1 - q^2(1 + zq)(1 + z^{-1}q)}.\]

Written as a series in $z$,

\[F(z, z^{-1}, q) = \sum_{k \in \mathbb Z} F_k(q) z^k\]

where $F_k(q)$ is keeping track of partitions with exactly $k$ more odd parts than even parts. We are thus interested in the expression

\begin{flalign*}
    \sum_{n = 0}^\infty \left(rd_o(n) - rd_e(n)\right)q^n = \sum_{k = 1}^\infty \left(F_k(q) - F_{-k}(q)\right).
\end{flalign*}

We will also define the ancillary $G_k(q)$ by

\[\frac{1}{1 - q^2(1 + zq)(1 + z^{-1}q)} = \sum_{k \in \mathbb Z} G_k(q) z^k\]

so that

\[F_k = qG_{k-1}(q) + q^2G_{k + 1}(q) + q^3G_k(q).\]

We can also see that $G_k(q) = G_{-k}(q)$ so that

\begin{flalign*}
    F_k(q) - F_{-k}(q) =& \left(qG_{k-1}(q) + q^2G_{k + 1}(q) + q^3G_k(q)\right) \\ 
                        &- \left(qG_{-k-1}(q) + q^2G_{-k + 1}(q) + q^3G_{-k}(q)\right) \\
                        =& \left(qG_{k-1}(q) + q^2G_{k + 1}(q) + q^3G_k(q)\right) \\ 
                        &- \left(qG_{k+1}(q) + q^2G_{k - 1}(q) + q^3G_k(q)\right) \\
                        =& (q-q^2)(G_{k - 1} - G_{k + 1}).
\end{flalign*}

This allows our series to telescope.

\begin{flalign*}
    \sum_{n = 0}^\infty \left(rd_o(n) - rd_e(n)\right)q^n &=  \sum_{k = 1}^\infty \left(F_k(q) - F_{-k}(q)\right) \\
    &= \sum_{k = 1}^\infty (q - q^2)(G_{k - 1} - G_{k + 1}) \\
    &= (q - q^2)(G_0 + G_1).
\end{flalign*}

Now we are left with determining $G_0$ and $G_1$. By geometric series,

\[\frac{1}{1 - q^2(1 + zq)(1 + z^{-1}q)} = \sum_{n = 0}^\infty q^{2n}(1 + zq)^n (1 + z^{-1}q)^n.\]

Which, by the binomial theorem, is

\[\sum_{n = 0}^\infty \sum_{i = 0}^n \sum_{j = 0}^n \binom ni \binom nj q^{2n + i + j}z^{i - j}.\]

We only need the constant and linear terms, so just take the terms corresponding to $i = j$ and $i = j + 1$ to find

\begin{flalign*}
    G_0(q) &= \sum_{n = 0}^\infty \sum_{j = 0}^n \binom{n}{j}^2 q^{2n + 2j}, \\
    G_1(q) &= \sum_{n = 0}^\infty \sum_{j = 0}^n \binom{n}{j + 1} \binom{n}{j} q^{2n + 2j + 1}.
\end{flalign*}

From which we can extract the following formula for the coefficients.

\begin{flalign*}
    [q^{2m}]G_0(q) &= \sum_{n = 0}^\infty \binom{n}{m - n}^2, \\
    [q^{2m - 1}]G_1(q) &= \sum_{n = 0}^\infty \binom{n}{m - n} \binom{n}{m - n - 1}.
\end{flalign*}

Notice that $G_0(q)$ has only even powers of $q$ and $G_1(q)$ only odd powers. The theorem follows from the fact $rd_o(n) - rd_e(n) = (q - q^2)(G_0 + G_1)$.
\end{proof}

We give recurrences for these differences whose proofs we only sketch, but can be verified computationally. Using the recurrences and mathematical induction, one can then prove Conjecture \ref{Parity Conjecture}.

\begin{lemma}
    For every $m \geq 0$, if we define $a_m := rd_o(2m + 1) - rd_e(2m + 1)$ then
    \begin{flalign*}
        (m+5)(4m^2+28m+29)a_{m+4} =& 4(2m^3+21m^2+58m+43)a_{m+1} \\
        &+(4m^3+44m^2+127m+117)a_{m+2} \\
        &+2(4m^3+46m^2+142m+107)a_{m+3} \\
        &-(m+1)(4m^2+36m+61)a_m.
    \end{flalign*}
    For every $m \geq 0$, if we define $b_m := rd_o(2m) - rd_e(2m)$ then
    \begin{flalign*}
        (m+5)(m^2+2m-2)b_{m+4} =& m(m+4)(2m+3)b_{m+1} \\
        &+(m+1)(m^2+5m+8)b_{m+2} \\
        &+(m+1)(m+4)(2m+3)b_{m+3} \\
        &-(m+1)(m^2+4m+1)b_m.
    \end{flalign*}
\end{lemma}

\begin{proof}[Sketch of proof]
    We can write $\sum_{k \in \mathbb Z} G_k(q) z^k$ by
    \[\frac{1}{1 - q^2 - q^4 - q^3(z + z^{-1})} = \frac{1}{1 - q^2 - q^4} \sum_{n = 0}^\infty \left(\frac{q^3}{1 - q^2 - q^4}\right)^n (z + z^{-1})^n.\]
    Extracting the constant coefficient using the known identity
    \[\sum_{r\geq0}\binom{2r}{r}t^r=\frac1{\sqrt{1-4t}},\]
    yields
    \[G_0(q)=\frac{1}{\sqrt{(1 - q^2 - q^4)^2 - 4q^6}}.\]
    and similarly the linear coefficient is given by
    \[G_1(q)=\frac{1 - q^2 - q^4 - \sqrt{(1 - q^2 - q^4)^2 - 4q^6}}{2q^3\sqrt{(1 - q^2 - q^4)^2 - 4q^6}}.\]
    So that
    \begin{flalign*}
        A(q) := \sum_{m = 0}^\infty a_m q^{2m+1} &= qG_0(q) - q^2G_1(q) \\
        &= \frac{q^4+3q^2-1+\sqrt{(1 - q^2 - q^4)^2 - 4q^6}}{2q\sqrt{(1 - q^2 - q^4)^2 - 4q^6}}.
    \end{flalign*}
    It is then routine to verify that $A(q)$ satisfies the differential equation
    \begin{flalign*}
        0={}&q^3(q^8-2q^6-q^4-2q^2+1)A'''(q)\\
        &+10q^4(2q^6-3q^4-q^2-1)A''(q)\\
        &+2q(39q^8-30q^6+5q^4+30q^2-11)A'(q)\\
        &+4(11q^8+7q^6-6q^4-2q^2-4)A(q)\\
        &-2q(10q^4-12q^2-19).
    \end{flalign*}
    The lemma is then proven by substituting the proper derivatives and collecting the coefficient in front of $q^{2(m + 4) + 1}$. We get the expression for $b_m$ with analogous manipulations.
\end{proof}

We now prove Conjecture \ref{Parity Conjecture} in the following two propositions.

\begin{proposition}
The sequence $(a_m)$ is positive for all $m$ and increasing for $m \geq 2$.
\end{proposition}

\begin{proof}
    Note the base cases $a_0 = a_1 = a_2 = 1, a_3 = 3,$ and $a_4 = 6$ are all positive. Assume that for some $m \geq 0$ we have
    \[0 < a_m \leq a_{m+1} \leq a_{m+2} \leq a_{m+3}.\]
    Subtracting $(m + 5)(4m^2 + 28m + 29)a_{m + 3}$ from both sides of the recurrence yields
    \begin{align*}
    (m&+5)(4m^2+28m+29)(a_{m+4}-a_{m+3})\\
    =&\left(4(2m^3+21m^2+58m+43)a_{m+1} - (m+1)(4m^2+36m+61)a_m\right)\\
    &+(4m^3+44m^2+127m+117)a_{m+2} +(4m^3+44m^2+115m+69)a_{m+3} \\
    \geq& \left(4(2m^3+21m^2+58m+43) -(m+1)(4m^2+36m+61)\right)a_m \\
    =&(4m^3+44m^2+135m+111)a_m>0.
    \end{align*}
    Hence
    \[(m+5)(4m^2+28m+29)(a_{m+4}-a_{m+3})>0\]
    and since $(m+5)(4m^2+28m+29)$ is positive for $m > 0$, so is $a_{m+4}$. The proposition thus follows by induction.
\end{proof}

\begin{proposition}
    The sequence $(b_m)$ is positive for all $m \geq 5$ and increasing for $m \geq 4$. Additionally $b_1$ is the only negative term.
\end{proposition}

\begin{proof}
    Note the base cases $b_1=-1, b_2 = b_3 = b_4 = 0,$ and $b_5 = 2$. Assume that for some $m \geq 2$ we have
    \[0 \leq b_m \leq b_{m+1} \leq b_{m+2} \leq b_{m+3}\]
    with $b_{m+3}>0$.
    Subtracting $(m+5)(m^2+2m-2)b_{m+3}$ from both sides of the recurrence yields
    \begin{align*}
    (m&+5)(m^2+2m-2)(b_{m+4}-b_{m+3}) \\
    =&\left(m(m+4)(2m+3)b_{m+1}-(m+1)(m^2+4m+1)b_m\right)\\
    &+(m+1)(m^2+5m+8)b_{m+2} +(m^3+6m^2+15m+22)b_{m+3} \\
    \geq& \left(m(m+4)(2m+3)-(m+1)(m^2+4m+1)\right)b_m \\
    =&(m^3+6m^2+7m-1)b_m > 0.
    \end{align*}
    Hence
    \[(m+5)(m^2+2m-2)(b_{m+4}-b_{m+3})>0\]
    and since $(m+5)(m^2+2m-2)$ is positive for $m \geq 2$, so is $b_{m+4}$. The proposition thus follows by induction. Note that for $m = 0$, the coefficient $m^3+6m^2+7m-1$ is negative.
\end{proof}

If we define $rd_0(n)$ as the coefficient of $q^n$ in the series $F_0(q)$, that is $rd_0(n)$ counts the number of partitions of perimeter $n$ into distinct parts such that there are an equal number of even parts and odd parts, we obtain the following combinatorial theorem.

\begin{theorem}
    For $n \geq 0$ we have
    \[rd_0(2n + 1) = \# \{H\text{'s} \text{ in peakless Motzkin paths of length }n\}\]
    and
    \[rd_0(2n + 2) = \# \{H\text{'s} + U\text{'s} \text{ in peakless Motzkin paths of length }n\}.\]
\end{theorem}

We note that the theorem implies that $rd_0(2n+2) - rd_0(2n + 1)$ counts the number of $U$'s in peakless Motzkin paths of length $n$ and thus is a positive quantity. We prove the theorem by defining a bijection between the set of the partitions counted on the left and the steps in the Motzkin paths counted on the right.

\begin{proof}
    We first take a partition, write it as a word in the letters $\{U, H, D\}$, cyclically permute the word using Lemma \ref{Raney}, and then delete the final letter which will be guaranteed to be a $D$, and obtain a peakless Motzkin path of the desired length. In the other direction we take a peakless Motzkin path with a designated part that we are counting, append a final $D$, cyclically rotate to put the designated part at the start, and finally we can undo the encoding process to get back parts of the partition forming a bijection.
    
    Let $\lambda$ be a partition with perimeter $2n + 1$ into $k$ distinct even parts and $k$ distinct odd parts. Notice that
    \[P(\lambda) := \lambda_1 + \ell(\lambda) - 1 = \lambda_1 + 2k - 1 = 2n + 1\]
    which readily implies that $\lambda_1 = 2(n - k + 1)$ is even. Let $W$ be an empty word. We delete $\lambda_1$ and append $H^*$, the designated $H$ step that we are counting, to $W$. Now for $1 \leq j \leq n - k$, if $2j$ is a part of $\lambda$ append a $U$ to $W$ and if not append an $H$. Notice that there are exactly $k - 1$ copies of the letter $U$ and $n - 2k + 1$ copies of the letter $H$. We will use the technique of stars and bars to keep track of the lengths of successive odd parts that appear and these will become the $D$ steps in our word. For $1 \leq j \leq n - k + 1$ if $2j - 1$ is a part, put a star and if it is not put a bar. We can now write the lengths of the number of stars as a weak composition of $k$ say $c = (c_1, c_2, \ldots, c_{n - 2k + 2})$ where $c_i$ is the number of stars, possibly 0, between $(i-1)$th bar and the $i$th. Starting at the designated $H^*$, put $c_i$ copies of the letter $D$ after the $i$th $H$. Notice that there are $k$ copies of $D$ and the total length of the word is $n + 1$. So if we define a sequence $b = (b_1, b_2, \ldots, b_{n + 1})$ such that $b_i$ is equal to $1, 0,$ or $-1$ if the $i$th letter of $W$ is a $U, H,$ or $D$ respectively, the total sum is $-1$ so we can apply Lemma \ref{Negative Raney} to obtain a cyclic shift of $b$ with all strict partial sums non-negative. This, in turn corresponds to a cyclic shift of $W$ such that we end with a $D$ step and no previous step goes below the $x$-axis. Delete this final $D$ to obtain a Motzkin path of length $n$. This Motzkin path is peakless since each $D$ step was inserted after an $H$ or a $D$.

    In the reverse direction, take a peakless Motzkin path of length $n$ and designate one of the $H$ steps to be counted as $H^*$. We append a $D$ step at the end and define $k$ to be the number of $D$ steps. Cyclically rotate $H^*$ to the front of the word. Create the weak composition $c = (c_1, c_2, \ldots, c_{n - 2k + 2})$ by defining $c_i$ to be the number of $D$ steps after the $i$th $H$. Define $\mu$ to be an empty partition and starting at 1 append to $\mu$ the smallest $c_1$ consecutive odd parts, skip the next odd part; then append the next $c_2$ consecutive odd parts, skip the next odd part et cetera. Continue this process for each $c_i$ through $c_{n - 2k + 1}$ until $\mu$ has all $k$ odd parts. Now delete all $D$ steps and append to $\mu$ a part of size $2j$ if the $j$th letter after $H^*$ is a $U$. Finally append to $\mu$ a part of size $2(n - k + 1)$. Then $\mu$ is a partition into distinct parts with perimeter $2n + 1$ that has an equal number of even parts and odd parts.
    
    For the even perimeter case $2n + 2$, let $\lambda$ be a partition with perimeter $2n + 2$ into $k$ distinct even parts and $k$ distinct odd parts. In this case,
    \[P(\lambda) := \lambda_1 + \ell(\lambda) - 1 = \lambda_1 + 2k - 1 = 2n + 2,\]
    which implies that $\lambda_1 = 2n - 2k + 3 = 2(n - k + 1) + 1$ is odd. Let $W$ be an empty word. Now for $1 \leq j \leq n - k + 1$, if $2j - 1$ is a part of $\lambda$ append a $U$ to $W$ and if not append an $H$ in either case designate the first letter. Notice that $W$ has $k - 1$ copies of the letter $U$ and $n - 2k + 2$ copies of the letter $H$. We again use stars and bars to keep track of the $k$ even parts. For $1 \leq j \leq n - k + 1$, if $2j$ is a part put a star and if it is not put a bar. We again define a weak composition $c = (c_1, c_2, \ldots, c_{n - 2k + 2})$ of $k$, where $c_i$ is the number of stars between the $(i-1)$th bar and the $i$th bar. Starting at the first $H$ step in $W$, put $c_i$ copies of the letter $D$ after the $i$th $H$. Notice that there are $k$ copies of $D$, $k - 1$ copies of $U$, and $n - 2k + 2$ copies of $H$, so the total length of the word is $n + 1$. Again we can apply Lemma \ref{Negative Raney} to obtain a unique cyclic shift of $W$ such that the strict partial sums are non-negative and the word ends with a $D$ step. Deleting this final $D$ yields a peakless Motzkin path of length $n$ with a designated step of either $U$ or $H$.
    
    In the reverse direction, take a peakless Motzkin path of length $n$ and designate one of the $U$ or $H$ steps to be counted as $U^*$ or $H^*$, respectively. We append a $D$ step at the end and define $k$ to be the number of $D$ steps. Cyclically rotate the word to bring the designated step to the front. Create the weak composition $c = (c_1, c_2, \ldots, c_{n - 2k + 2})$ by defining $c_i$ to be the number of $D$ steps after the $i$th $H$. Define $\mu$ to be an empty partition, and starting at 2 append the smallest $c_1$ even parts and continue like in the odd perimeter case. Next, delete all $D$ steps from the rotated word to obtain a word of length $n - k + 1$. For $1 \leq j \leq n - k + 1$, if the $j$th letter is a $U$, append $2j - 1$ to $\mu$, giving $k - 1$ odd parts. Lastly append the part $2n - 2k + 3$ to $\mu$. Then $\mu$ is a partition into distinct parts with perimeter $2n + 2$ having an equal number of even parts and odd parts.
\end{proof}

\begin{example}
    Consider the partition $\lambda = (20, 15, 14, 13, 12, 9, 8, 7, 5, 4, 2, 1)$ which has perimeter 31. The largest part is even so is becomes our designated $H^*$. Appending the steps corresponding to the even parts we have
    \[W = H^*UUHUHUUHH.\]
    The odd parts have the corresponding stars and bars
    \[*|***|**||\]
    which corresponds to the composition $(1, 3, 2, 0, 0)$. We then append the corresponding steps of size $D$ to obtain
    \[W = H^*DUUHDDDUHDDUUHH.\]
    The corresponding cyclic rotation is given by
    \[UUHHH^*DUUHDDDUHD[D].\]
    Finally deleting the bracketed $D$ gives a peakless Motzkin path of length 15 as desired.
\end{example}

\begin{example}
    Consider the Motzkin path of length 12 with designated $U^*$ given by $UHUHHDHU^*HHDD$. We append a $D$ to the end and cycle the designated part to the start to obtain
    \[U^*HHDDDUHUHHDH.\]
    The $D$'s correspond to the stars and bars diagram
    \[|***|||*|\]
    and thus $\mu$ has even parts of size $4, 6, 8,$ and $14$. Deleting the $D$'s gives
    \[U^*HHUHUHHH\]
    and so $\mu$ has odd parts of size $1, 7, 11$ and the largest part of size $2n - 2k + 3 = 24 - 8 + 3 = 19$. Thus $\mu = (19, 14, 11, 8, 7, 6, 4, 1)$.
\end{example}

\end{document}